\documentclass[11pt]{amsart}
\usepackage{amsmath,amssymb,epsfig,color, amsthm, mathrsfs}
\usepackage{wrapfig,lipsum,floatrow,sidecap}
\usepackage{graphicx}
\usepackage{multirow}
\usepackage{hhline}
\usepackage{url}
\graphicspath{ {images/} }
\allowdisplaybreaks

\newtheorem{theorem}{Theorem}[section]

\newtheorem{definition}[theorem]{Definition}
\newtheorem{lemma}[theorem]{Lemma}

\newtheorem{example}[theorem]{Example}

\newcommand{\vanish}[1]{}\parskip=12pt

\def\p{\prime}
\def\pp{{\prime\prime}}
\def\b{\textbf{B}}

\def\D{\mathcal{D}}

\def\K{\mathcal{K}}

\numberwithin{equation}{section}

\begin{document}
\title{Braid Index Bounds Ropelength From Below}
\author{Yuanan Diao}
\address{Department of Mathematics and Statistics\\
University of North Carolina Charlotte\\
Charlotte, NC 28223}
\email{ydiao@uncc.edu}

\begin{abstract}
For an un-oriented link $\K$, let $L(\K)$ be the ropelength of $\K$. It is known that when $\K$ has more than one component, different orientations of the components of $\K$ may result in different braid index. We define the largest braid index among all braid indices corresponding to all possible orientation assignments of $\K$ the {\em absolute braid index} of $\K$ and denote it  by $\b(\K)$. In this paper, we show that
there exists a constant $a>0$ such that $L(\K)\ge a \b(\K) $ for any $\K$, {\em i.e.}, the ropelength  of any link is bounded below by its absolute braid index (up to a constant factor). 
\end{abstract} 

\maketitle
\section{Introduction}\label{s1}

An important geometric property of a link is its ropelength, defined (intuitively) as the minimum length of a unit thickness rope that can be used to tie the link. Let $\K$ be an un-oriented link, $Cr(\K)$ be the minimum crossing number of $\K$ and $L(\K)$ be the ropelength of $\K$. One way to understand the ropelength of a link is to associate it with the topological complexity of the link as measured by some link invariant. For example one can attempt to express the ropelength, or an estimate of it, of a link as a function of the minimum crossing number of the link. This turned out to be a very difficult problem in general and results are limited. For example, while it has been shown that $L(\K)\ge 31.32$ for any nontrivial knot $\K$ \cite{Denne2006}, the precise ropelength for any given nontrivial knot is not known. It has been shown in \cite{Buck, Buck2} that in general $L(\K)\ge 1.105 (Cr(\K))^{3/4}$ and that this $3/4$ power can be attained by a family of infinitely many links \cite{Cantarella1998, Diao1998}. On the other hand, not all links obtain this $3/4$ power law since there exist families of infinitely many links such that the ropelength of a link from any of these families grows linearly as the crossing number of the link \cite{Diao2003}. This result is based on the fact that the ropelength of a link is bounded below by the bridge number of the link (multiplied by some positive constant) and the fact that there are families of (infinitely many) links whose bridge numbers are proportional to their crossing numbers. To the knowledge of the author, the bridge number is the only known link invariant that has been used to establish the ropelength of a link. Of course, if a link has a small bridge number, then we would not be able to establish a good lower bound of the ropelength of the link using its bridge number. In this paper we show that the braid index of a link can also be used to bound the ropelength of the link from below (again up to the multiple of a positive constant). 

For an un-oriented link $\K$ with more than one component, different orientations of the components of $\K$ may result in different braid index (an invariant of oriented links). We will call the largest braid index among all braid indices corresponding to different orientation assignments of $\K$ the {\em absolute braid index} of $\K$ and denote it by
$\b(\K)$. In this paper, we show that there exists a constant $a>0$ such that $L(\K)\ge a \b(\K) $ for any $\K$, {\em i.e.}, the ropelength  of any link is bounded below by its absolute braid index (up to a constant factor). Since the bridge number of a link is smaller than or equal to its absolute braid index, and many links with bounded bridge numbers can have absolute braid indices proportional to their crossing numbers, this result will allow us to establish better ropelength lower bound for many more links.

\section{Special cord diagrams and their Seifert diagrams}\label{s2}

\begin{definition}\label{cord}
{\em
Let $\K$ be an oriented link and $\D$ be a projection diagram of $\K$. Without loss of generality we will assume that the projection plane is $z=0$. Let $\alpha_1$, $\alpha_2$, ..., $\alpha_n$ be simple arcs of $\D$. We say that $\alpha_1$, $\alpha_2$, ..., $\alpha_n$ form a {\em special cord diagram} $R$ (associated with $\D$) if the following conditions hold: (i) the end points of $\alpha_1$, $\alpha_2$, ..., $\alpha_n$ do not cross each other and are distributed on a topological circle $C$ (in the projection plane $z=0$);  (ii) the interiors of $\alpha_1$, $\alpha_2$, ..., $\alpha_n$ are completely within the disk $\overline{C}$ bounded by $C$; (iii) $\D\setminus \cup_{1\le j\le n}\alpha_j$ does not intersect $\overline{C}$; (iv) the arc $\gamma_j$ on $\D$ corresponding to $\alpha_j$ resides in a slab $Z_j$ defined by $z_j^\p\le z\le z_j^{\pp}$ for some $z_j^\p\le z_j^\pp$; (v) $Z_k\cap Z_j=\emptyset$ if $j\not=k$.}
\end{definition}

Notice that by conditions (iv) and (v), a new special cord diagram $R^\p$ can be obtained from a cord diagram $R$ by replacing each $\alpha_j$ with a simple curve $\alpha_j^\p$: the choice of $\alpha_j^\p$ is arbitrary so long as it is the projection of a curve $\gamma_j^\p$ that resides within the slab $Z_j$ sharing the same end points with $\gamma_j$ and is bounded within $C$. The result is still a special cord diagram associated with $\D^\p$ where $\D^\p$ is the resulting new projection diagram which is still a projection diagram of $\K$. We say that $R^\p$ is {\em equivalent} to $R$. In other word, the cords have fixed end points but otherwise can move freely within $C$.

\begin{definition}\label{Seifert_diagram}{\em 
A Seifert diagram of a special cord diagram is the diagram obtained from the special cord diagram by smoothing all crossings in the diagram.}
\end{definition}

Note: since we are only interested in the Seifert diagrams of special cord diagrams, the over/under strands of the crossings in the diagrams are not important to us and will not be shown in our figures. Also, in a Seifert diagram of a special cord diagram, there are only two types of curves: topological circles (Seifert circles) and simple curves  with their end points on $C$ (we will call these {\em partial Seifert circles}). See Figure \ref{cord_diagram} for an illustration of a special cord diagram and its Seifert diagram.

\begin{figure}[htb!]
\includegraphics[scale=1.0]{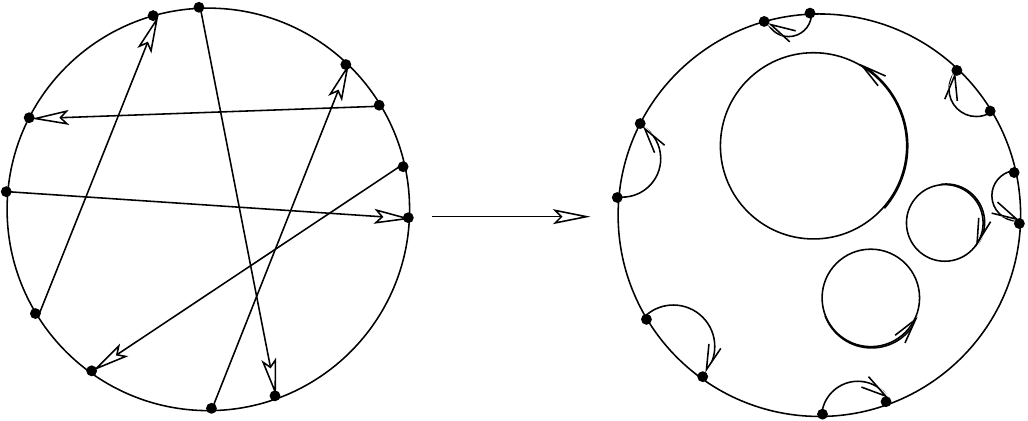}
\caption{Left: A special cord diagram; Right: The Seifert diagram of it.
\label{cord_diagram}}
\end{figure} 

Let us assign $C$ an (arbitrary) orientation. Consider an oriented simple curve $\beta$ with its end points on $C$ and its interior bounded within $C$. We call  the arc $\beta^\p$ of $C$ that shares end points with $\beta$ and is parallel to $\beta$ (in terms of their orientations) the {\em companion} of $\beta$ and the region bounded by $\beta$ and $\beta^\p$ the {\em domain} of $\beta$.

\begin{definition}\label{coherent}{\em 
A special cord diagram $R$ is said to be {\em coherent} if we can choose an orientation of $C$ such that the Seifert diagram of $R$ satisfies the following conditions:
(i) its Seifert circles (if there are any) are concentric to each other and all share the same orientation with $C$; (ii) the domain of any partial Seifert circle cannot contain any Seifert circles; (iii) if the domain of a a partial Seifert circle contains another partial Seifert circle, it must contain the entire domain of that partial Seifert circle.}
\end{definition}

The special cord diagram as shown in Figure \ref{cord_diagram} is not coherent: no matter how we choose the orientation of $C$, there is always a partial Seifert circle whose domain contains some Seifert circles. Figure \ref{coherent} shows a coherent special cord diagram that is equivalent to it. The following lemma assures us that this is always possible.

\begin{figure}[htb!]
\includegraphics[scale=1.0]{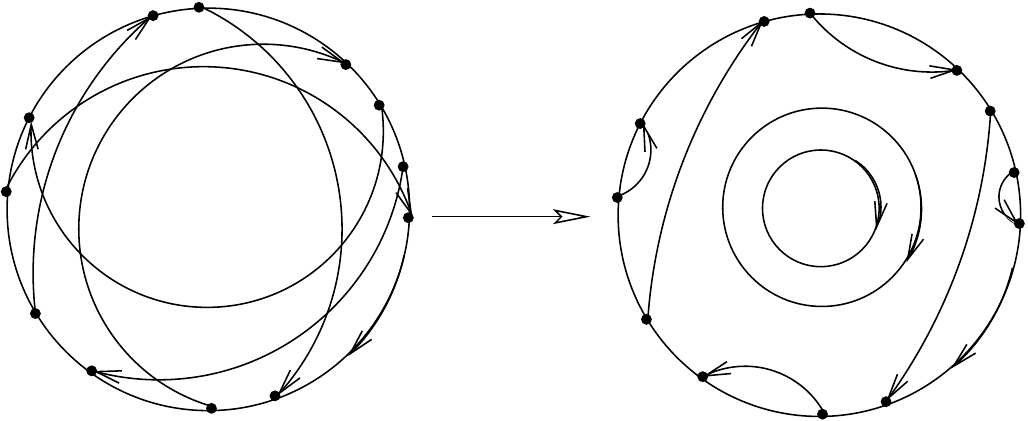}
\caption{Left: A special cord diagram equivalent to the special cord diagram shown in the left of Figure \ref{cord_diagram}; Right: The corresponding Seifert diagram is coherent with the orientation $C$ as shown in the figure.
\label{coherent}}
\end{figure}

\medskip
\begin{lemma}\label{Lemma1}
Let $R$ be a special cord diagram with $n$ cords, then there exists a coherent special cord diagram $R^\p$ that is equivalent to $R$. Furthermore, the Seifert diagram of $R^\p$ contains exactly $n$ partial Seifert circles and at most $n-1$ Seifert circles.
\end{lemma}

\begin{proof}
Let us assign $C$ the clockwise orientation. We will prove the lemma by induction. The case of $n=1$ is trivial. Assume that the statement of the lemma holds for $n=n_0\ge 1$ and let us consider the case for $n=n_0+1$. Consider first the special cord diagram $R_{n_0}$ containing the first $n_0$ cords. By the induction assumption, there exists a coherent special cord diagram $R^\p_{n_0}$ that is equivalent to $R_{n_0}$ such that its Seifert diagram $S_{n_0}$ contains $n_0$ partial Seifert circles and at most $n_0-1$ (concentric) Seifert circles which all have clockwise orientation. We will construct $R^\p_{n_0+1}$ by choosing an appropriate $\alpha_{n_0+1}^\p$ (namely the last cord appropriately modified) starting from $R^\p_{n_0}$.  

There are two cases to consider. In the first case, the intersection of the companion of $\alpha_{n_0+1}$ (the last cord of $R_{n_0+1}$) with the companion of any other partial Seifert circle is either empty or a simply connected arc on $C$.  Figure \ref{intersect2} illustrates how $\alpha_{n_0+1}^\p$ may be chosen and the resulting Seifert diagram after all crossings have been smoothed. Notice that in the illustration we only showed Seifert circles and partial Seifert circles of $R^\p_{n_0+1}$. Although $\alpha_{n_0+1}^\p$ may have additional crossings with cords in the original diagram $R^\p_{n_0}$, once these crossings are smoothed, due to the orientation of the curves involved, it is easy to verify that the resulting Seifert circles and partial Seifert circles are as illustrated in Figure \ref{intersect2}. It is clear that in this case we obtain a new coherent Seifert diagram with one additional partial Seifert circle and no additional Seifert circles. Thus the statement of the lemma holds for this case. In the second case, the intersection of the companion of $\alpha_{n_0+1}$ with the companion of at least one other partial Seifert circle consists of two disconnected simple arcs on $C$ as shown in the left of Figure \ref{new_Seifert}. The middle of Figure \ref{new_Seifert} shows how 
$\alpha_{n_0+1}^\p$ is constructed and the right side shows the resulting Seifert diagram: it has one additional partial Seifert circle and one additional Seifert circle. Again the statement of the lemma holds and this proves the lemma.
\end{proof}

\begin{figure}[htb!]
\includegraphics[scale=1.0]{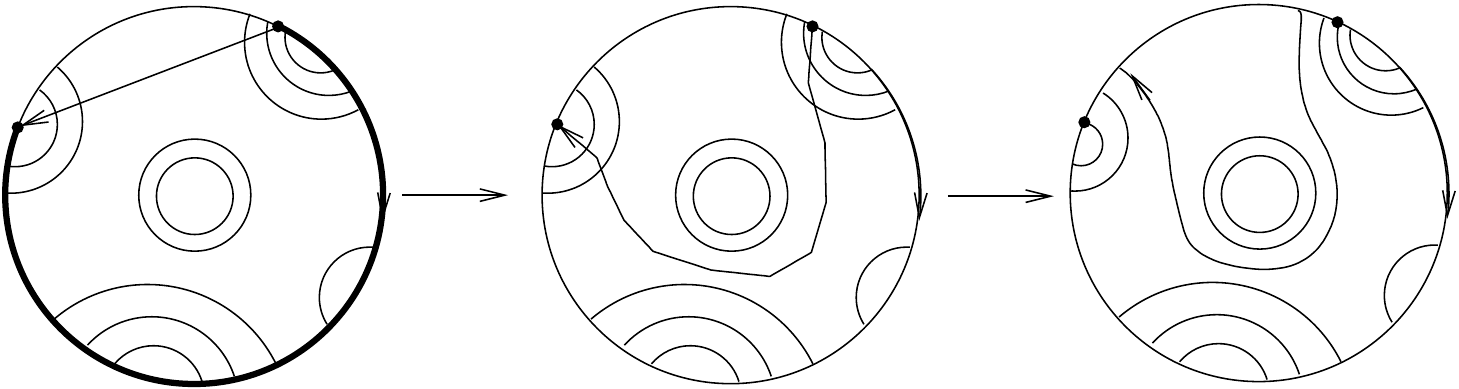}
\caption{Left: The companion of $\alpha_{n_0+1}$ is shown in thick line and its intersection with the companions of other partial Seifert circles are either empty or a simply connected arcs on $C$. The orientations of the Seifert circles and partial Seifert circles are parallel to that of $C$ (not shown in the figure); Middle: The choice of $\alpha^\p_{n_0+1}$; Right: The resulting Seifert diagram.
\label{intersect2}}
\end{figure}

\begin{figure}[htb!]
\includegraphics[scale=1.0]{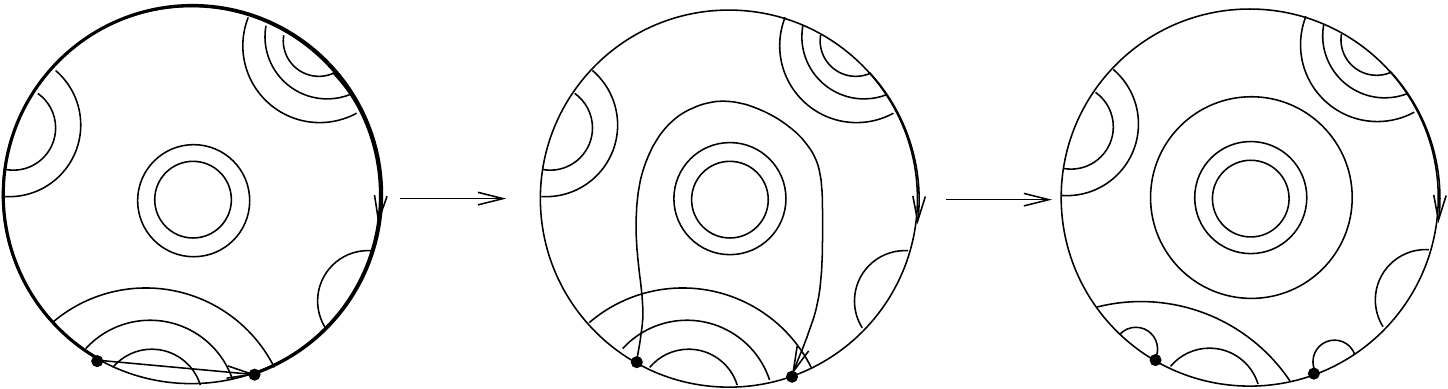}
\caption{Left: The companion of $\alpha_{n_0+1}$ is shown in thick line. Notice that its intersection with the companion of one partial Seifert circles consists of two disjoint arcs on $C$. The orientations of the Seifert circles and partial Seifert circles are parallel to that of $C$ (not shown in the figure); Middle: The choice of $\alpha^\p_{n_0+1}$; Right: The resulting Seifert diagram.
\label{new_Seifert}}
\end{figure}

\section{Absolute braid index bounds the ropelength from below}\label{s3}

Let us first consider links realized on the cubic lattice. Let $\K$ be an un-oriented link and $K_c$ a realization of $\K$ on the cubic lattice. The length of $K_c$ is denoted by $L(K_c)$ and the minimum of $L(K_c)$ over all lattice realization $K_c$ of $\K$ is called the {\em minimum step number} of $\K$ and is denoted by $L_c(\K)$. One nice property of $L_c(\K)$ is that in theory it can be determined through exhaustive search. For example, it has been shown that $L_c(\K)=24$, $30$ and $34$ for the trefoil \cite{Diao1993}, the figure 8 knot and $34$ for the $5_1$ knot \cite{Scharein2009}. However in reality the precise value of $L_c(\K)$ is also very difficult to determine and the above three examples are the only known results for nontrivial knots in fact. A line segment on $K_c$ between two neighboring lattice points is called a {\em step}. A step that is parallel to the $x$-axis is called an $x$-step. $y$-steps and $z$-steps are similarly defined. Let $x(K_c)$,  $y(K_c)$ and $z(K_c)$ be the total number of $x$-steps, $y$-steps and $z$-steps respectively, then $x(K_c)+y(K_c)+z(K_c)=L(K_c)$. Without loss of generality, let us assume that $z(K_c)\ge \max\{x(K_c),y(K_c)\}$ hence $z(K_c)\ge (1/3)L(K_c)$ and $x(K_c)+y(K_c)=L(K_c)-z(K_c)\le (2/3)L(K_c)$. We now consider the projection of $K_c$ to the $xy$-plane. The resulting diagram is not a regular one. However if we tilt $K_c$ slightly, then we will obtain a regular projection of $K_c$ and all crossings will occur near a lattice point on the $xy$-plane. At a lattice point where we see crossings of the projection, consider the arcs of the projection bounded by a unit square centered at the lattice point as shown in Figure \ref{square}. It is rather obvious that these arcs define a special cord diagram with each arc resides in a slab that is disjoint from other slabs that contain the other arcs, since each cord consists of two half  steps that are parts of some $x$ and/or $y$ steps, and possibly some consecutive $z$ steps, hence two different cords is separated by a slab of thickness near one (without the tilt it would be precisely one).

\begin{figure}[htb!]
\includegraphics[scale=.8]{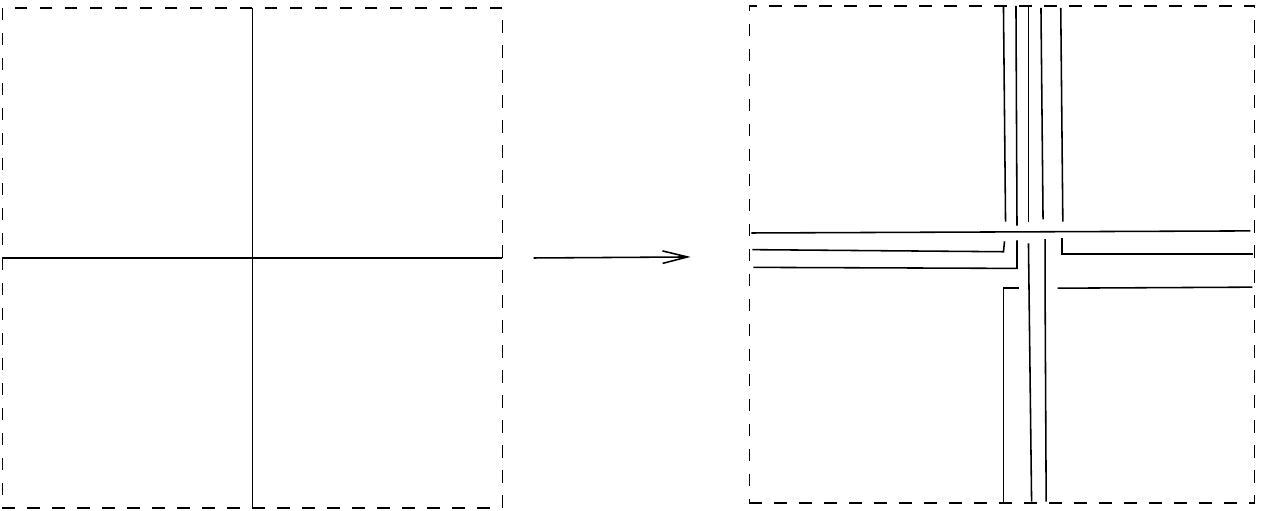}
\caption{Left: the top view of a unit length square centered at a lattice point where the projected strands of $K_c$ intersect; Right: A slightly tilted projection leads to a special cord diagram.
\label{square}}
\end{figure}

Let $m$ be the number of  lattice points where the projection of $K_c$ has intersections, and let $n_j$ be the number of arcs involved at the $j$-th such lattice point. Now assign $K_c$ an orientation so that it yields $\b(\K)$. By Lemma \ref{Lemma1}, we can modify the special cord diagrams to make them coherent. The result is a regular projection $K^\p$ which is an ambient isotopy of $K_c$. After we smooth all crossings in $K^\p$, at the $j$-th cord diagram, we obtain $n_j$ partial Seifert circles and at most $n_j-1$ Seifert circles. Each partial Seifert circles and each arc of $K^\p$ that is not contained in these special cord diagrams must be connected to at least one other partial Seifert circle in order to form a complete Seifert circle, thus the total number of Seifert circles in $K^\p$ formed by the partial Seifert circles and the arcs not in the cord diagrams is at most $\frac{1}{2}\sum_{1\le j\le m}n_j$. It follows that the total number of Seifert circles in $K^\p$ (denoted by $s(K^\p)$) is bounded above by
$\frac{1}{2}\sum_{1\le j\le m}n_j+\sum_{1\le j\le m}(n_j-1)<\frac{3}{2}\sum_{1\le j\le m}n_j$. On the other hand, each cord in the special diagram has total length one in its $x$ and $y$-step portion, hence the total length of the $x$ and $y$-steps in the projection of $K_c$ is at least $\sum_{1\le j\le m}n_j$. Thus we have $\sum_{1\le j\le m}n_j\le x(K_c)+y(K_c)\le (2/3)L(K_c)$ and it follows that 
$$
s(K^\p)< \frac{3}{2}\sum_{1\le j\le m}n_j\le  \frac{3}{2} \frac{2}{3}L(K_c)=L(K_c).
$$
It is well known that for any oriented link diagram $\D$, we have $\textbf{b}(\D)\le s(\D)$ where $s(\D)$ is the number of Seifert circles in $\D$ \cite{Ya}. Since $K_c$ has the orientation that yields $\textbf{b}(K_c)=\b(\K)$, we have $\b(\K)=\textbf{b}(K_c)=\textbf{b}(K^\p)\le s(K^\p)<L(K_c)$. Since $K_c$ is arbitrary, replacing it by a step length minimizer of $\K$ yields $\b(\K)<L_c(\K)$. Finally, it has been shown that $L_c(\K)<14L(\K)$ \cite{Diao2002}, thus we have proven the following theorem:

\begin{theorem}\label{T1}
Let $\K$ be an un-oriented link, then $\b(\K)<L_c(\K)<14L(\K)$, that is, $L(\K)>(1/14)\b(\K)$.
\end{theorem}

In a recent paper, the author and his colleagues derived explicit formulas for braid indices of many alternating links including all  alternating Montesinos links \cite{Diao2019}. Using these formulas one can easily identify many families of alternating links with small bridge numbers but with braid indices proportional to their crossings numbers, these provide us new examples of link families whose ropelengths grow at least linearly as their crossing numbers (since the previously known method based on the bridge numbers would not get us these results). The following are just a few such examples.

\begin{example}\label{E1}{\em 
Let $\K$ be the $(2,2n)$ torus link, a two component link with $2n$ crossings. There are two different choices for the orientations of the two components. One of them yields a braid index of $2$ while the other yields a braid index of $n+1$. Thus we have $\b(\K)=n+1=Cr(\K)/2+1$, hence $L_c(\K)>n+1$ and $L(\K)>(n+1)/14>Cr(\K)/28$.}
\end{example}

\begin{example}\label{E1}{\em 
Let $\K$ be a twist knot with $n\ge 4$ crossings. We have $\b(\K)=\textbf{b}(\K)=k+1=(Cr(\K)+1)/2$ if $n=2k+1$ is odd, and $\b(\K)=\textbf{b}(\K)=k+2=Cr(\K)/2+1$ if $n=2k+2$ is even. It follows that $L(\K)>(Cr(\K)+1)/28$ for any twist knot $\K$.}
\end{example}

\begin{figure}[htb!]
\includegraphics[scale=.8]{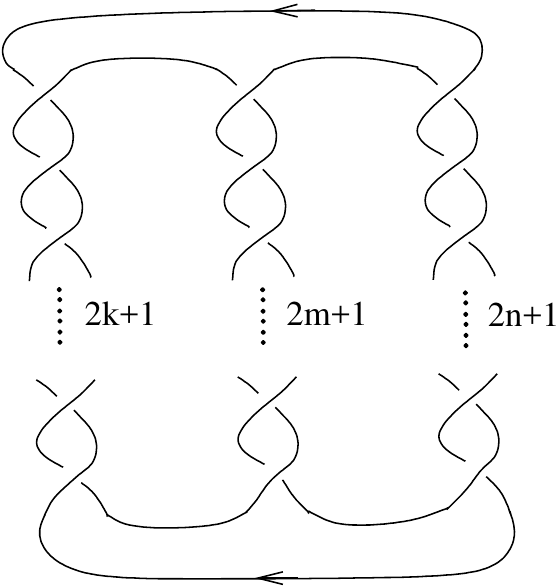}
\caption{An alternating pretzel knot with three columns containing $2k+1$, $2m+1$ and $2n+1$ crossings respectively ($k$, $m$ and $n$ are non-negative integers and the case of $k=m=n=0$ gives the trefoil knot). 
\label{Pretzel}}
\end{figure}

\begin{example}\label{E2}{\em 
Consider the pretzel knot $\K$ a projection of which is given in Figure \ref{Pretzel}. $Cr(\K)=2(k+m+n)+3$ since it is alternating. It can be calculated from the formulas given in \cite{Diao2019} that $\b(\K)=\textbf{b}(\K)=2+k+m+n>(1/2)Cr(\K)$. It follows that $L(\K)>Cr(\K)/28$ as well.}
\end{example}

Notice that in the above examples, the bridge numbers are either 2 or 3. Furthermore, since the link diagrams given in the above examples are all algebraic link diagrams, it is known that the ropelengths of these links grow at most linearly as their crossings numbers \cite{Diao2006}. Thus the ropelengths of these links in fact grow linearly as their crossing numbers.

\section{Further discussions}\label{s4}

For an oriented link $\K$ with a projection diagram $\D$, consider the HOMFLY-PT polynomial $H(\D,z,a)$ defined using the skein relation $aH(\D_+,z,a)-a^{-1}H(\D_-,z,a)=zH(\D_0,z,a)$ (and the initial condition $H(\D,z,a)=1$ if $\D$ is the trivial knot). Let $E(\D)$ and $e(\D)$ be the highest and lowest powers of $a$ in $H(\D,z,a)$ and define $\textbf{b}_0(\K)=(E(\D)-e(\D))/2+1$. It is a well known result that $\textbf{b}_0(\K)\le \textbf{b}(\K)$ where $\textbf{b}(\K)$ is the braid index of $\K$ \cite{Morton1986}. In the case that $\K$ is un-oriented, similarly to the definition of $\b(\K)$, we define 
$
\b_0(\K)=\max\{\textbf{b}_0(\K^\p): \ \K^\p\in O(\K)\}$ where $O(\K)$ is the set of oriented links obtained by assigning all possible orientations to the components of $\K$. Apparently we have $\b_0(\K)\le \b(\K)$ hence we have the following theorem, which is handy when we do not have a precise formula for the braid index of the link.

\begin{theorem}\label{T2}
Let $\K$ be an un-oriented link, then $\b_0(\K)<L_c(\K)<14L(\K)$ and $L(\K)>(1/14)\b_0(\K)$.
\end{theorem}

It has been conjectured that the ropelength of an alternating link $\K$ is bounded below by a constant multiple of its crossing number. Our result shows that this conjecture holds for many alternating links. A remaining challenge is about the alternating links whose absolute braid index is small, for example the $(2, 2n+1)$ torus knot whose braid index is 2. While its minimum projection looks so much like the minimum projection of the $(2, 2n)$ torus link and it is quite plausible that its ropelength should behave linearly as its crossing number, we do not have a way to prove it! We end this paper with this problem as a challenge to our reader.

\end{document}